\documentclass[12pt]{article}
\usepackage{a4}
\usepackage[fleqn]{amsmath}
\usepackage{amsthm}
\usepackage{amssymb}
\usepackage{xypic}

\evensidemargin\oddsidemargin

\newtheorem{proposition}{Proposition}
\newtheorem{corollary}[proposition]{Corollary}
\newtheorem{theorem}[proposition]{Theorem}
\newtheorem{lemma}[proposition]{Lemma}

\theoremstyle{definition}

\newtheorem{remark}{Remark}
\newtheorem{conjecture}{Conjecture}

\newcommand{\B}[1]{\mathbf{#1}}

\newcommand{\g}{\mathfrak g}
\newcommand{\h}{\mathfrak h}
\newcommand{\n}{\mathfrak{n}}

\DeclareMathOperator{\id}{id}

\DeclareMathOperator{\Ker}{Ker} \DeclareMathOperator{\Image}{Im}
 \DeclareMathOperator{\End}{End}
\DeclareMathOperator{\spanv}{Span} 
 
\DeclareMathOperator{\Hom}{Hom} \DeclareMathOperator{\fin}{fin}
 
 \DeclareMathOperator{\red}{red}
 \DeclareMathOperator{\height}{ht}

\let\kk\Bbbk
\let\alb\allowbreak
\def\lc{,\alb\ldots,\alb}

\def\C{\mathbb C}
\def\N{\mathbb N}
\def\Q{\mathbb Q}

\def\S{\mathbb S}
\def\Z{\mathbb Z}

\def\one{\mathbf 1}
\def\Rb{\mathbf R}

\def\kg{\mathfrak k}

\def\U{U}
\def\Uc{\check U_q\g}

\def\ftext#1{{\let\thefootnote\relax
\footnotetext{\vskip-.9\baselineskip\noindent #1}}}

\begin{document}
\title{Equivariant quantization of Poisson homogeneous spaces and Kostant's
problem} \author{E.~Karolinsky, A.~Stolin, and V.~Tarasov} \date{}

\maketitle


\begin{sloppy}

\vskip-.6\baselineskip
\hrule height0pt
\thispagestyle{empty}

\begin{abstract}
We find a partial solution to the longstanding problem of Kostant concerning
description of the so-called locally finite endomorphisms of highest weight
irreducible modules. The solution is obtained by means of its reduction to a
far-reaching extension of the quantization problem. While the classical
quantization problem consists in finding $\star$-product deformations of the
commutative algebras of functions, we consider the case when the initial
object is already a noncommutative algebra, the algebra of functions within
$q$-calculus.
\smallskip

\noindent{\bf Mathematics Subject Classifications (2000).} 17B37, 17B10,
53D17, 53D55.

\smallskip

\noindent{\bf Key words:} quantized universal enveloping algebra, Kostant's
problem, highest weight module, equivariant quantization, reduced fusion
element.
\end{abstract}

\section{Introduction}

The present paper is a closing paper of the series of papers written by the
authors on relations between Poisson homogeneous spaces and their
quantization, solutions of the dynamical Yang-Baxter equation and Kostant's
problem \cite{KS, KMST, KST_0, KST}.

Poisson homogeneous spaces were introduced by Drinfeld, and their relations
with the so-called classical doubles \cite{D_poiss_lie} were explained in
\cite{D_poiss_hom_spaces}. Using that, the first author classified Poisson
homogeneous spaces of quasi-triangular type in terms of Lagrangian
subalgebras of $\g \times \g$, where $\g$ is a semisimple complex
finite-dimensional Lie algebra \cite{K_compl}.

Later, in \cite{Lu_dyn}, Lu discovered a strong similarity between
classification of the Poisson homogeneous spaces of \cite{K_compl} and
classification of trigonometric solutions of the classical dynamical
Yang-Baxter equation (CDYBE) of \cite{Sch}. This similarity was explained in
\cite{KS}. Furthermore, the authors gave a classification of Poisson
homogeneous spaces of triangular type, and, developing Lu's method, they
proved that there is a natural one-to-one correspondence between Poisson
homogeneous spaces of triangular type and rational solutions of the CDYBE.

The latter result opened a way for an explicit quantization of certain
Poisson homogeneous spaces of triangular type using methods of
\cite{ES_DYBE} (one should notice that the existence of quantization of
Poisson manifolds was proved by Kontsevich in \cite{Konts}).

This idea was realized in \cite{KMST, KST_0, KST} and first relations
between quantization of Poisson homogeneous spaces of triangular type with
the so-called Kostant's problem for the highest weight irreducible
$U(\g)$-modules were considered in \cite{KST}. It is worth to mention that,
in particular, this approach enabled the authors to obtain explicit formulas
for quantization of the Kirillov-Kostant-Lie Poisson bracket on reductive
co-adjoint orbits of $\g$.

The main goal of the present paper is to develop methods for an explicit
quantization of Poisson homogeneous spaces of quasi-triangular type and its
relations with Kostant's problem.

It turns out that in order to get explicit formulas in the quasi-triangular
case, one has to work, instead of the standard universal enveloping algebra
of $\g$ and the standard algebra of regular functions on the corresponding
Lie group $G$, with their quantized versions. As a consequence, it turns out
that the latter quantization problems are related to Kostant's problem for
the quantum universal enveloping algebra.

To be more precise, let $\check{U}_q\g$ be the quantized universal
enveloping algebra ``of simply connected type'' \cite{Joseph_book} that
corresponds to a finite dimensional split semisimple Lie algebra $\g$. Let
$L(\lambda)$ be the irreducible highest weight $\Uc$-module of highest
weight $\lambda$. The aim of this paper is to show that for certain values
of $\lambda$, the action map $\Uc\to\bigl(\End L(\lambda)\bigr)_{\fin}$ is
surjective. Here $(\End L(\lambda)\bigr)_{\fin}$ stands for the locally
finite part of $\End L(\lambda)$ with respect to the adjoint action of
$\Uc$. For the Lie-algebraic case $(q=1)$, this problem is known as the
classical Kostant's problem, see~\cite{Ja, Jo, Maz, M-S1}. The complete
answer to it is still unknown even in the $q=1$ case. However, there are
examples of $\lambda$ for which the action map $U(\g)\to\bigl(\End
L(\lambda)\bigr)_{\fin}$ is not surjective. Such examples exist even in the
case $\g$ is of type $A$ \cite{M-S2}.

The main idea of our approach to Kostant's problem, both in the
Lie-algebraic and quantum group cases, is that $\bigl(\End
L(\lambda)\bigr)_{\fin}$ has two other presentations. First, it follows from
the results of \cite{KST} that $\bigl(\End L(\lambda)\bigr)_{\fin}$ is
canonically isomorphic to $\Hom_U\bigl(L(\lambda),L(\lambda)\otimes
F\bigr)$, where $U$ is $U(\g)$ (resp.~$\Uc$), and $F$ is the algebra of
(quantized) regular functions on the connected simply connected algebraic
group $G$ corresponding to the Lie algebra $\g$. In other words, $F$ is
spanned by matrix elements of finite dimensional representations of $U$ with
an appropriate multiplication.

One more presentation of the algebra $\bigl(\End L(\lambda)\bigr)_{\fin}$
comes from the fact that $\Hom_U\bigl(L(\lambda),L(\lambda)\otimes F\bigr)$
is isomorphic as a vector space to a certain subspace $F'$ of $F$. The
subspace $F'$ can be equipped with a $\star$-multiplication obtained from
the multiplication on $F$ by applying the so-called reduced fusion element.
Then $\bigl(\End L(\lambda)\bigr)_{\fin}$ is isomorphic as an algebra to
$F'$ with this new multiplication. For certain values of $\lambda$, the same
$\star$-multiplication on $F'$ can be defined by applying the universal
fusion element, that yields the affirmative answer to Kostant's problem in
such cases.

More exactly, consider the triangular decomposition $U=U^-U^0U^+$. We have
$L(\lambda)=M(\lambda)/K_\lambda\one_\lambda$, where $M(\lambda)$ is the
corresponding Verma module, $\one_\lambda$ is the generator of $M(\lambda)$,
and $K_\lambda\subset U^-$. Consider also the opposite Verma module
$\widetilde{M}(-\lambda)$ with the lowest weight $-\lambda$ and the lowest
weight vector $\widetilde\one_{-\lambda}$. Then its maximal $U$-submodule is of
the form $\widetilde{K}_{\lambda}\cdot\widetilde\one_{-\lambda}$, where
$\widetilde{K}_{\lambda}\subset U^+$.
We have $F'=F[0]^{K_\lambda+\widetilde{K}_\lambda}$ --- the subspace of
$U^0$-invariant elements of $F$ annihilated by both $K_\lambda$ and
$\widetilde{K}_\lambda$. The $\star$-product on
$F[0]^{K_\lambda+\widetilde{K}_\lambda}$ has the form
\[f_1\star_\lambda f_2=
\mu\left(J^{\red}(\lambda)(f_1\otimes f_2)\right),
\]
where $\mu$ is the multiplication on $F$, and the reduced fusion element
$J^{\red}(\lambda)\in U^-\,\widehat{\otimes}\,U^+$ is computed in terms of
the Shapovalov form on $L(\lambda)$. Notice that for generic $\lambda$ the
element $J^{\red}(\lambda)$ is equal up to an $U^0$-part to the fusion
element $J(\lambda)$ related to the Verma module $M(\lambda)$, see for
example~\cite{ES_DYBE}.

We also investigate limiting properties of $J(\lambda)$. In particular, for
some values of $\lambda_0$ we can guarantee that $f_1\star_\lambda
f_2\rightarrow f_1\star_{\lambda_0}f_2$ as $\lambda\rightarrow\lambda_0$.
Also, for any $\lambda_0$ having a ``regularity property'' of this kind, the
action map $U\to(\End L(\lambda_0))_{\fin}$ is surjective. This gives the
affirmative answer to the (quantum version of) Kostant's problem.

For some values of $\lambda$, the subspace
$F[0]^{K_\lambda+\widetilde{K}_\lambda}$ is a subalgebra of $F[0]$,
and can be considered as (a flat deformation of) the algebra of regular
functions on some Poisson homogeneous space $G/G_1$. In those cases,
the algebra $\bigl(F[0]^{K_\lambda+\widetilde{K}_\lambda},\star_\lambda\bigr)$
is an equivariant quantization of the Poisson algebra of regular functions on
$G/G_1$.

This paper is organized as follows. In Section~\ref{Sect-U} we recall the
definition of the version of quantized universal enveloping algebra used in
this paper, and some related constructions that will be useful in the
sequel.
In Section~\ref{Sect-main} we construct an isomorphism
$\Hom_U\bigl(L(\lambda),L(\lambda)\otimes F\bigr)\simeq
F[0]^{K_\lambda+\widetilde{K}_\lambda}$ and, as a corollary, provide a
construction of a star-product on $F[0]^{K_\lambda+\widetilde{K}_\lambda}$
in terms of the Shapovalov form on $L(\lambda)$. 
In Section~\ref{Sect-limit} we study limiting properties of fusion elements
and the corresponding star-products. Namely, in
Subsection~\ref{Subsect-regularity-general} we introduce the notion of a
$J$-regular weight. We prove that in a neighborhood of a $J$-regular weight
the fusion element behaves nicely, and give a solution to the Kostant's
problem for such weights (see Proposition \ref{prop-Kostant}). In
Subsections \ref{Subsect-limit-one-root} and \ref{Subsect-limit-new} we
provide non-trivial examples of $J$-regular weights. Finally, in Subsection
\ref{Subsect-symm} we apply limiting properties of fusion elements to
quantize explicitly certain Poisson homogeneous spaces (see Theorem
\ref{thm-lim}).

\subsection*{Acknowledgments} The authors are grateful to Maria Gorelik,
Jiang-Hua Lu, and Catharina Stroppel for useful discussions on the topic of
the paper.

\section{Algebra $\Uc$}\label{Sect-U}

Let $\kk$ be the field extension of $\C(q)$ by all fractional powers $q^{1/n}$,
$n\in\N=\{1,2,3,\ldots\}$. We use $\kk$ as the ground field.

Let $(a_{ij})$ a finite type $r\times r$ Cartan matrix. Let $d_i$ be
relatively prime positive integers such that $d_ia_{ij}=d_ja_{ji}$.
For any positive integer $k$, define
\[
[k]_i=\frac{q^{kd_i}-q^{-kd_i}}{q^{d_i}-q^{-d_i}}\,,\qquad
[k]_i!=[1]_i\,[2]_i\,\ldots\,[k]_i\,.
\]

The algebra $\U=\Uc$ is generated by the elements $t_i$, $t_i^{-1}$, $e_i$,
$f_i$, $i=1\lc r$, subject to the relations
\begin{gather*}
t_it_i^{-1}=t_i^{-1}t_i=1\,,\\
t_ie_jt_i^{-1}=q^{d_i\delta_{ij}}e_j,\\
t_if_jt_i^{-1}=q^{-d_i\delta_{ij}}f_j,\\
e_if_j-f_je_i=\delta_{ij}\,\frac{k_i-k_i^{-1}}{q^{d_i}-q^{-d_i}},
\ \mathrm{where}\ k_i=\prod_{j=1}^rt_j^{a_{ij}},\\
\sum_{m=0}^{1-a_{ij}}
\frac{(-1)^m}{[m]_i!\,[1-a_{ij}-m]_i!}e_i^ke_je_i^{1-a_{ij}-k}=0\
\textrm{for}\ i\neq j,\\
\sum_{m=0}^{1-a_{ij}}
\frac{(-1)^m}{[m]_i!\,[1-a_{ij}-m]_i!}f_i^kf_jf_i^{1-a_{ij}-k}=0\
\textrm{for}\ i\neq j
\end{gather*}
Notice that $k_ie_jk_i^{-1}=q^{d_ia_{ij}}e_j$,
$k_if_jk_i^{-1}=q^{-d_ia_{ij}}f_j$.

\medskip
The algebra $\U$ is a Hopf algebra with the comultiplication $\Delta$,
the counit $\varepsilon$, and the antipode $\sigma$ given by
\[
\begin{array}{lll}
\Delta(t_i)=t_i\otimes t_i,&\varepsilon(t_i)=1,&\sigma(t_i)=t_i^{-1}\\
\Delta(e_i)=e_i\otimes1+k_i\otimes e_i,&\varepsilon(e_i)=0,&\sigma(e_i)=-k_i^{-1}e_i\\
\Delta(f_i)=f_i\otimes k_i^{-1}+1\otimes f_i,&\varepsilon(f_i)=0,&\sigma(f_i)=-f_ik_i.
\end{array}
\]
In what follows we will sometimes use the Sweedler notation for
comultiplication.

Let $U^0$ be the subalgebra of $\U$ generated by the elements $t_1\lc t_r$,
$t_1^{-1}\lc t_r^{-1}$. Let $U^+$ and $U^-$ be the subalgebras generated
respectively by the elements $e_1\lc e_r$ and $f_1\lc f_r$. We have
a triangular decomposition $U=U^-U^0U^+$. Denote by $\theta$ the involutive
automorphism of $U$ given by $\theta(e_i)=-f_i$, $\theta(f_i)=-e_i$,
$\theta(t_i)=t_i^{-1}$. Notice that $\theta$ gives an algebra isomorphism
$U^-\to U^+$. Set $\omega=\sigma\theta$, i.e., $\omega$ is the involutive
antiautomorphism of $U$ given by $\omega(e_i)=f_ik_i$,
$\omega(f_i)=k_i^{-1}e_i$, $\omega(t_i)=t_i$.

Let $(\h, \Pi, \Pi^\vee)$ be a realization of $(a_{ij})$ over $\Q$, that is,
$\h$ is (a rational form of) a Cartan subalgebra of the corresponding
semisimple Lie algebra, $\Pi=\{\alpha_1\lc \alpha_r\}\subset\h^*$ the set of
simple roots, $\Pi^\vee=\{\alpha_1^\vee\lc \alpha_r^\vee\}\subset\h$ the set
of simple coroots. Let $\Rb$ be the root system, $\Rb_+$ the set of positive
roots, and $W$ the Weyl group. Denote by $s_\alpha\in W$ the reflection
corresponding to a root $\alpha$. For $w\in W$, $\lambda\in\h^*$ we set
$w\cdot\lambda=w(\lambda+\rho)-\rho$. Let $u_1\lc u_l\in\h$ be the simple
coweights, i.e., $\langle\alpha_i,u_j\rangle=\delta_{ij}$. We denote by
$\rho$ the half sum of the positive roots.
For $w\in W$ and $\lambda\in\h^*$, we set
$w\cdot\lambda=w(\lambda+\rho)-\rho$.

Let
\[
Q_+=\sum_{\alpha\in\Pi}\Z_+\alpha\,.
\]
For $\beta=\sum_jc_j\alpha_j\in Q_+$, denote
$\height\beta=\sum_jc_j\in\mathbb Z_+$. For $\lambda, \mu\in\h^*$ we set
$\lambda\geq\mu$ iff $\lambda-\mu\in Q_+$.

Take an invariant scalar product $(\,\cdot\,|\,\cdot\,)$ on $\h^*$ such that
$(\alpha|\alpha)=2$ for any short root $\alpha$.
Then $d_i=\frac{(\alpha_i|\alpha_i)}{2}$.

Denote by $T$ the multiplicative subgroup generated by $t_1\lc t_r$.
Any $\lambda\in\h^*$ defines a character $\Lambda:T\to\kk$ given by
$t_i\mapsto q^{d_i\langle\lambda,u_i\rangle}$. We will write
$\Lambda=q^\lambda$.
Notice that $q^\lambda(k_i)=q^{d_i\langle\lambda,\alpha^\vee_i\rangle}$.
We extend $q^\lambda$ to the subalgebra $U^0$ by linearity.
We say that an element $x\in\U$ is of weight $\lambda$ if
$txt^{-1}=q^\lambda(t)x$ for all $t\in T$.

\smallskip
For a $\U$-module $V$, we denote by
\[
V[\lambda]=\{v\in V\,|\, tv=q^\lambda(t)v \mathrm{\ for\ all\ } t\in T\}
\]
the weight subspace of weight $\lambda$. We call the module
$V$ admissible if $V$ is a direct sum of finite-dimensional weight subspaces
$V[\lambda]$.

\smallskip
The Verma module $M(\lambda)$ over $\U$ with highest weight $\lambda$ and
highest weight vector $\one_\lambda$ is defined in the standard way:
\[
M(\lambda)=U^-\one_\lambda,\qquad U^+\one_\lambda=0,\qquad
t\one_\lambda=q^\lambda(t)\one_\lambda,\quad t\in T.
\]
The map $U^-\to M(\lambda)$, $y\mapsto y\one_\lambda$ is an isomorphism
of $U^-$-modules.

\smallskip
Set $U_+^\pm=\Ker\varepsilon|_{U^\pm}$ and denote by $x\mapsto(x)_0$ the
projection $U\to U^0$ along $U_+^-\cdot U+U\cdot U_+^+$. For any
$\lambda\in\h^*$ consider $\pi_\lambda:U^+\otimes U^-\to\kk$,
$\pi_\lambda(x\otimes y)=q^\lambda((\sigma(x)y)_0)$, and $\mathbb
S_\lambda:U^-\otimes U^-\to\kk$, $\S_\lambda(x\otimes
y)=\pi_\lambda(\theta(x)\otimes y)=q^\lambda((\omega(x)y)_0)$. We call
$\S_\lambda$ the Shapovalov form on $U^-$ corresponding to $\lambda$. We can
regard $\S_\lambda$ as a bilinear form on $M(\lambda)$.

Set
\[
K_\lambda=\{y\in U^-\,|\,\pi_\lambda(x\otimes y)=0\ \mbox{\rm
for all}\ x\in U^+\},
\]
\[
\widetilde{K}_\lambda=\{x\in U^+\,|\,\pi_\lambda(x\otimes y)=0\
\mbox{\rm for all}\ y\in U^-\}.
\]
Clearly, $K_\lambda$ is the kernel of $\S_\lambda$,
$\widetilde{K}_\lambda=\theta(K_\lambda)$. Notice that
$K(\lambda)=K_\lambda\cdot\one_\lambda$ is the largest proper submodule
of $M(\lambda)$, and $L(\lambda)=M(\lambda)/K(\lambda)$ is the irreducible
$U$-module with highest weight $\lambda$. Denote by
$\overline\one_\lambda$ the image of $\one_\lambda$ in $L(\lambda)$.

The following propositions are well known for the Lie-algebraic case. They
also hold for the case of $\U=\Uc$. Proposition \ref{prop-first} follows
from a simple $U_q \mathfrak{sl}(2)$ computation. For Propositions
\ref{prop-second}, \ref{prop-third}, \ref{prop-fourth}, see
\cite{Joseph_book}.

\begin{proposition}\label{prop-first}
Assume that $\lambda\in\h^*$ satisfies
$\langle\lambda+\rho,\alpha^\vee_i\rangle=n\in\N$ for a simple root $\alpha_i$.
Then $f_i^n$ is in $K_\lambda$.
\end{proposition}

\begin{proposition}\label{prop-second}
Let $\lambda\in\h^*$ be dominant integral.
Then the $\U$-module $L(\lambda)$ is finite dimensional, and
\ $\dim\bigl(L(\lambda)\bigr)[\mu]=\dim\bigl(L(\lambda)\bigr)[w\mu]$
for any $\mu\in\h^*$ and $w\in W$.
\end{proposition}

\begin{proposition}\label{prop-third}
Assume that $\lambda\in\h^*$ satisfies
$\langle\lambda+\rho,\alpha^\vee\rangle=n\in\N$ for some $\alpha\in\Rb_+$
and $\langle\lambda+\rho,\beta^\vee\rangle\not\in\N$ for all
$\beta\in\Rb_+\setminus\{\alpha\}$. Then $K_\lambda$ is generated by
a single element of weight $-n\alpha$.
\end{proposition}

\begin{proposition}\label{prop-fourth}
Let $\lambda\in\h^*$ be dominant integral,
$\langle\lambda+\rho,\alpha^\vee_i\rangle=n_i$, $i=1\lc r$.
Then $K_\lambda$ is generated by the elements $f_i^{n_i}$, $i=1\lc r$.
\end{proposition}




In the sequel we need some properties of the universal $R$-matrix of $\U$. Namely, let $V_1$, $V_2$
be $\U$-modules such that $V_1$ is a direct sum of highest weight modules
or $V_2$ is a direct sum of lowest weight modules.
Then the $R$-matrix induces an isomorphism
$\check{R}:V_1\otimes V_2\to V_2\otimes V_1$ of $\U$-modules.
Moreover, if $V_1$ is a highest weight module
with highest weight $\lambda$ and highest weight vector $\one_\lambda$, and
$V_2$ is a lowest weight module with lowest weight $\mu$ and lowest weight
vector $\widetilde{\one}_\mu$, then
$\check{R}(\one_\lambda\otimes\widetilde{\one}_\mu)=
q^{-(\lambda|\mu)}\widetilde{\one}_\mu\otimes\one_\lambda$.

Let $F=\kk[G]_q$ be the quantized algebra of regular functions on a
connected simply connected algebraic group $G$ that corresponds
to the Cartan matrix $(a_{ij})$ (see \cite{Joseph_book, KS_alg_func}). We
can consider $F$ as a Hopf subalgebra in the dual Hopf algebra $U^\star$.
We will use the left and right regular actions of $\U$ on $F$ defined
respectively by the formulae $(\overrightarrow{a}f)(x)=f(xa)$ and
$(f\overleftarrow{a})(x)=f(ax)$. Notice that $F$ is a sum of
finite-dimensional admissible $\U$-modules with respect to both 
regular actions of $\U$ (see \cite{KS_alg_func}).

\section{Star products and fusion elements}\label{Sect-main}

\subsection{Algebra of intertwining operators}\label{Subsect-main-general}


Let us denote by $U_{\fin}\subset U$ the subalgebra of locally finite
elements with respect to the right adjoint action of $\U$ on itself. We will
use similar notation for any (right) $\U$-module.

For any (left) $\U$-module $M$
we equip $F$ with the left regular $\U$-action and consider
the space $\Hom_U(M,M\otimes F)$. For any $\varphi,\psi\in\Hom_U(M,M\otimes
F)$ define
\begin{equation}\label{eqn-basic-star-prod}
\varphi \ast \psi =(\id\otimes \mu)\circ (\varphi \otimes \id)\circ \psi,
\end{equation}
where $\mu$ is the multiplication in $F$. We have
$\varphi*\psi\in\Hom_U(M,M\otimes F)$, and this definition equips
$\Hom_U(M,M\otimes F)$ with a unital associative algebra structure.

Consider the map $\Phi:\Hom_U(M,M\otimes F)\to\End M$, $\varphi\mapsto
u_\varphi$, defined by $u_\varphi(m)=(\id\otimes\varepsilon)(\varphi(m))$;
here $\varepsilon(f)=f(1)$ is the counit in $F$. Consider $U_{\fin}$,
$\Hom_U(M,M\otimes F)$ and $\End M$ as right $\U$-module algebras:
$U_{\fin}$ via right adjoint action, $\Hom_U(M,M\otimes F)$ via right
regular action on $F$ (i.e., $(\varphi\cdot a)(m)=(\id\otimes
\overleftarrow{a})(\varphi(m))$), and $\End M$ in a standard way (i.e.,
$u\cdot a=\sum_{(a)}\sigma(a_{(1)})_Mu{a_{(2)}}_M$). Then $\Hom_U(M,M\otimes
F)_{\fin}=\Hom_U(M,M\otimes F)$, and $\Phi: \Hom_U (M, M\otimes
F)\longrightarrow(\End M)_{\fin}$ is an isomorphism of right $\U$-module
algebras (see \cite[Proposition 6]{KST}).

Now we apply this to $M=M(\lambda)$ and $M=L(\lambda)$. Since
$U_{\fin}\to(\End M(\lambda))_{\fin}$ is surjective (see \cite{Joseph_book,
Joseph_Letzter_Verma}), we have the following commutative diagram
\[
\xymatrix{\Hom_U (M(\lambda), M(\lambda)\otimes F)
\ar[d]^{\Phi_{M(\lambda)}} \ar[r] & \Hom_U
(L(\lambda), L(\lambda)\otimes F) \ar[d]^{\Phi_{L(\lambda)}} \\
(\End M(\lambda))_{\fin} \ar[r] & (\End L(\lambda))_{\fin}}
\]
(see \cite[Proposition 9]{KST}).

For any
$\varphi\in\Hom_U(L(\lambda),L(\lambda)\otimes F)$ the formula
$\varphi(\overline\one_\lambda) =
\overline{\one}_\lambda \otimes f_\varphi + \sum_{\mu<\lambda} v_\mu \otimes
f_\mu$, where $v_\mu$ is of weight $\mu$,
defines a map $\Theta:\Hom_U(L(\lambda),L(\lambda)\otimes
F)\to F[0]$, $\varphi\mapsto f_\varphi$.

\begin{theorem}\label{thm-theta}
$\Theta$ is an embedding, and its image equals $F[0]^{K_\lambda+\widetilde{K}_\lambda}$.
\end{theorem}

To prove Theorem \ref{thm-theta} we need some preparations.

In the sequel $V$ stands for an $\U$-module which is a direct sum
of finite dimensional admissible $\U$-modules.

For an admissible $\U$-module $M$
we will denote by $M^*$ its restricted
dual.

Let $\widetilde{M}(\lambda)$ be the ``opposite Verma module'' with
the lowest weight $\lambda\in\h^*$ and the lowest weight vector
$\widetilde{\one}_\lambda$. It is clear that
$\widetilde{K}_{-\lambda}\cdot\widetilde{\one}_\lambda$ is
the largest proper submodule in $\widetilde{M}(\lambda)$.

\begin{lemma}\label{lem-n+-inv}
$\Hom_{U^-}(M(\lambda),V)=(V\otimes M(\lambda)^*)^{U^-}$,
$\Hom_{U^+}(\widetilde{M}(\lambda),V)=(V\otimes\widetilde{M}(\lambda)^*)^{U^+}$.
\end{lemma}

\begin{proof}
For any $\varphi\in\Hom_{U^-}(M(\lambda),V)$ the image of
$\varphi$ is equal to the finite-dimensional $U^-$-submodule
$U^-\varphi(\one_\lambda)$. Therefore for any $x\in
U^-$ such that $x\one_\lambda$ is a weight vector whose
weight is large enough we have
$\varphi(x{\one}_\lambda)=x\varphi(\one_\lambda)=0$. Thus $\varphi$
corresponds to an element in $(V\otimes M(\lambda)^*)^{U^-}$.

The second part of the lemma can be proved similarly.
\end{proof}

Choose vectors $\one_{\lambda}^*\in M(\lambda)^*[-\lambda]$
and $\widetilde{\one}_{-\lambda}^*\in
\widetilde{M}(-\lambda)^*[\lambda]$ such that
$\langle\one_{\lambda}^*,\one_{\lambda}\rangle=
\langle\widetilde{\one}_{-\lambda}^*,\widetilde{\one}_{-\lambda}\rangle=1$.
Define maps
$\zeta:\Hom_U(M(\lambda),V\otimes\widetilde{M}(-\lambda)^*)\to
V[0]$ and
$\widetilde{\zeta}:\Hom_U(\widetilde{M}(-\lambda),V\otimes M(\lambda)^*)\to V[0]$ by the formulae
$\varphi(\one_{\lambda})=\zeta_\varphi\otimes\widetilde{\one}_{-\lambda}^*\,+$
lower order terms,
$\varphi(\widetilde{\one}_{-\lambda})=
\widetilde{\zeta}_\varphi\otimes\one_{\lambda}^*\,+$ higher
order terms.

Consider also the natural maps
\begin{gather*}
r:\Hom_U(M(\lambda)\otimes
\widetilde{M}(-\lambda),V)\to\Hom_U(M(\lambda),V\otimes
\widetilde{M}(-\lambda)^*),\\
\widetilde{r}:\Hom_U(M(\lambda)\otimes
\widetilde{M}(-\lambda),V)\to\Hom_U(\widetilde{M}(-\lambda),V\otimes M(\lambda)^*).
\end{gather*}

\begin{proposition}\label{prop-zeta-r-isomorphisms}
Maps $\zeta$, $\widetilde{\zeta}$, $r$, and $\widetilde{r}$ are
vector space isomorphisms, and the diagram
\[
\xymatrix{ \Hom_U(M(\lambda)\otimes\widetilde{M}(-\lambda),V)
\ar[d]^{\check{R}^{-1}} \ar[r]^{r} &
\Hom_U(M(\lambda),V\otimes\widetilde{M}(-\lambda)^*) \ar[r]^{\hskip
2.1cm\zeta} &
V[0] \ar[d]^{q^{-(\lambda|\lambda)}}\\
\Hom_U(\widetilde{M}(-\lambda)\otimes M(\lambda),V) \ar[r]^{\widetilde{r}} &
\Hom_U(\widetilde{M}(-\lambda),V\otimes M(\lambda)^*)\ar[r]^{\hskip
2.1cm\widetilde{\zeta}} & V[0]}
\]
is commutative.
\end{proposition}

\begin{proof}
First of all notice that we have the natural
identification
\begin{gather*}
\Hom_U(M(\lambda),V\otimes\widetilde{M}(-\lambda)^*)=(V\otimes\widetilde{M}(-\lambda)^*)^{U^+}[\lambda],
\end{gather*}
Further on, we have
\begin{gather*}
\Hom_U(M(\lambda)\otimes
\widetilde{M}(-\lambda),V)=\Hom_U(M(\lambda),\Hom(\widetilde{M}(-\lambda),V))=\\
\Hom_{U^+}(\widetilde{M}(-\lambda),V)[\lambda]=V[0].
\end{gather*}
On the other side,
$\Hom_{U^+}(\widetilde{M}(-\lambda),V)=(V\otimes\widetilde{M}(-\lambda)^*)^{U^+}$
by Lemma \ref{lem-n+-inv}. Now it is clear that the map $r$ (resp.\ $\zeta$) corresponds to the
identification
$\Hom_U(M(\lambda)\otimes\widetilde{M}(-\lambda),V)=(V\otimes\widetilde{M}(-\lambda)^*)^{U^+}[\lambda]$
(resp.\ $(V\otimes\widetilde{M}(-\lambda)^*)^{U^+}[\lambda]=V[0]$).

The second part of the proposition concerning $\widetilde{r}$ and
$\widetilde{\zeta}$ can be verified similarly.

Finally, since
$\check{R}^{-1}(\widetilde{\one}_{-\lambda}\otimes\one_\lambda)=
q^{-(\lambda|\lambda)}\one_\lambda\otimes\widetilde{\one}_{-\lambda}$,
the whole diagram is commutative.
\end{proof}

Now note that the pairing $\pi_\lambda:U^+\otimes U^-\to\kk$ naturally defines a pairing
$\widetilde{M}(-\lambda)\otimes M(\lambda)\to\kk$. Denote by
$\chi_\lambda:M(\lambda)\to\widetilde{M}(-\lambda)^*$ the
corresponding morphism of $\U$-modules. The kernel of
$\chi_\lambda$ is equal to
$K(\lambda)=K_\lambda\cdot\mathbf1_\lambda$, and the image of
$\chi_\lambda$ is
$(\widetilde{K}_\lambda\cdot\widetilde\one_{-\lambda})^\perp\simeq L(\lambda)$.
Therefore $\chi_\lambda$ can be naturally represented as
$\chi_\lambda''\circ\chi_\lambda'$, where
\[
M(\lambda)\stackrel{\chi_\lambda'}{\longrightarrow}L(\lambda)
\stackrel{\chi_\lambda''}{\longrightarrow}\widetilde{M}(-\lambda)^*.
\]

The morphisms $\chi_\lambda'$ and $\chi_\lambda''$ induce the
commutative diagram of embeddings
\[
\xymatrix{\Hom_U(L(\lambda),V\otimes L(\lambda)) \ar[d] \ar[r] &
\Hom_U(M(\lambda),V\otimes L(\lambda)) \ar[d]\\
\Hom_U(L(\lambda),V\otimes \widetilde{M}(-\lambda)^*) \ar[r] &
\Hom_U(M(\lambda),V\otimes \widetilde{M}(-\lambda)^*).}
\]
It is clear that the following lemma holds:

\begin{lemma}\label{lem-image-chi-lambda-pr}
The image of $\Hom_U(L(\lambda),V\otimes L(\lambda))$ in
$\Hom_U(M(\lambda),V\otimes\widetilde{M}(-\lambda)^*)$ under the
embedding above consists of the morphisms
$\varphi:M(\lambda)\to V\otimes\widetilde{M}(-\lambda)^*$ such
that $\varphi(K_\lambda\one_{\lambda})=0$ and
$\varphi(M(\lambda))\subset V\otimes(\widetilde{K}_\lambda\widetilde{\one}_{-\lambda})^\perp$.\qed
\end{lemma}

\begin{proposition}\label{prop-tilde-K-lambda-0}
Let $\varphi\in\Hom_U(M(\lambda),V\otimes\widetilde{M}(-\lambda)^*)$.
Then $\varphi(M(\lambda))\subset V\otimes(\widetilde{K}_\lambda\widetilde{\one}_{-\lambda})^\perp$
iff $\widetilde{K}_\lambda\zeta_\varphi=0$.
\end{proposition}

\begin{proof}
First notice that $\varphi(M(\lambda))\subset
V\otimes(\widetilde{K}_\lambda\widetilde{\one}_{-\lambda})^\perp$ iff
$\varphi(\mathbf1_\lambda)\in
V\otimes(\widetilde{K}_\lambda\widetilde{\one}_{-\lambda})^\perp$.
Indeed, for any $x\in U$ we have
$\varphi(x\mathbf1_\lambda)=\sum_{(x)}(x_{(1)}\otimes
x_{(2)})\varphi(\mathbf1_\lambda)$ and
$U\cdot(\widetilde{K}_\lambda\widetilde{\one}_{-\lambda})^\perp=
(\widetilde{K}_\lambda\widetilde{\one}_{-\lambda})^\perp$.

Denote by $\psi$ the element in
$\Hom_{U^+}(\widetilde{M}(-\lambda),V)$ that corresponds to
$\varphi(\mathbf1_\lambda)\in(V\otimes\widetilde{M}(-\lambda)^*)^{U^+}$
(see Lemma \ref{lem-n+-inv}). Under this notation
$\varphi(\mathbf1_\lambda)\in
V\otimes(\widetilde{K}_\lambda\widetilde{\one}_{-\lambda})^\perp$ iff
$\psi(\widetilde{K}_\lambda\widetilde{\one}_{-\lambda})=0$.
On the other hand,
$\zeta_\varphi=\psi(\widetilde{\one}_{-\lambda})$ and
$\psi(\widetilde{K}_\lambda\widetilde{\one}_{-\lambda})=
\widetilde{K}_\lambda\psi(\widetilde{\one}_{-\lambda})=\widetilde{K}_\lambda\zeta_\varphi$.
This completes the proof.
\end{proof}

\begin{proposition}\label{prop-K-lambda-0}
Let $\varphi\in\Hom_U(M(\lambda),V\otimes\widetilde{M}(-\lambda)^*)$.
Then $\varphi(K_\lambda\one_{\lambda})=0$ iff
$K_\lambda\zeta_\varphi=0$.
\end{proposition}

\begin{proof}
Consider
$\widehat{\varphi}=r^{-1}(\varphi)\in\Hom_U(M(\lambda)\otimes
\widetilde{M}(-\lambda),V)$,
$\overline{\varphi}=\widehat{\varphi}\circ\check{R}^{-1}\in\Hom_U(\widetilde{M}(-\lambda)\otimes
M(\lambda),V)$, and
$\widetilde{\varphi}=\widetilde{r}(\overline{\varphi})\in\Hom_U(\widetilde{M}(-\lambda),V\otimes
M(\lambda)^*)$ (see Proposition \ref{prop-zeta-r-isomorphisms}).
Since $K_\lambda\one_{\lambda}\otimes\widetilde{M}(-\lambda)$ is
an $U\otimes U$-submodule in $M(\lambda)\otimes\widetilde{M}(-\lambda)$, one has
$\varphi(K_\lambda\one_{\lambda})=0$ iff
$\widehat{\varphi}(K_\lambda\one_{\lambda}\otimes\widetilde{M}(-\lambda))=0$
iff $\overline{\varphi}(\widetilde{M}(-\lambda)\otimes K_\lambda\one_{\lambda})=0$
iff $\widetilde{\varphi}(\widetilde{M}(-\lambda))
\subset V\otimes(K_\lambda\one_{\lambda})^\perp$.

Arguing as in the proof of Proposition \ref{prop-tilde-K-lambda-0} we see
that $\widetilde{\varphi}(\widetilde{M}(-\lambda)) \subset
V\otimes(K_\lambda\one_{\lambda})^\perp$ iff
$K_\lambda\widetilde{\zeta}_{\widetilde{\varphi}}=0$. Now it is enough to
notice that $\widetilde{\zeta}_{\widetilde{\varphi}}=
q^{-(\lambda|\lambda)}\zeta_\varphi$ by Proposition
\ref{prop-zeta-r-isomorphisms}, and therefore
$K_\lambda\widetilde{\zeta}_{\widetilde{\varphi}}=0$ iff
$K_\lambda\zeta_\varphi=0$.
\end{proof}

Define maps $u:\Hom_U(L(\lambda),L(\lambda)\otimes V)\to V[0]$ and
$v:\Hom_U(L(\lambda),V\otimes L(\lambda))\to V[0]$
via $\varphi\mapsto u_\varphi$, where
$\varphi(\overline\one_\lambda)=\overline\one_\lambda\otimes
u_\varphi\,+$ lower order terms, and $\psi\mapsto v_\psi$, where
$\psi(\overline\one_\lambda)=
v_\psi\otimes\overline\one_\lambda\,+$ lower order terms.

\begin{proposition}\label{prop-u}
The map $v$ defines the isomorphism
$\Hom_U(L(\lambda),V\otimes L(\lambda))\simeq
V[0]^{K_\lambda+\widetilde{K}_\lambda}$.
\end{proposition}

\begin{proof}
Observe that $v$ can be decomposed as
\[
\Hom_U(L(\lambda),V\otimes L(\lambda))\longrightarrow
\Hom_U(M(\lambda),V\otimes\widetilde{M}(-\lambda)^*)
\stackrel{\zeta}{\longrightarrow}V[0],
\]
where the first arrow is the natural embedding considered in Lemma
\ref{lem-image-chi-lambda-pr}. Now it is enough to apply the above
mentioned lemma and Propositions \ref{prop-tilde-K-lambda-0} and
\ref{prop-K-lambda-0}.
\end{proof}

\begin{proposition}\label{prop-v}
The diagram
\[
\xymatrix{\Hom_U(L(\lambda),V\otimes L(\lambda)) \ar[d]^{\check{R}}
\ar[r]^{\hskip 1.7cm v} &
V[0] \ar[d]^{\id} \\
\Hom_U(L(\lambda),L(\lambda)\otimes V) \ar[r]^{\hskip 1.7cm u} & V[0]}
\]
is commutative.
\end{proposition}

\begin{proof}
Take $\psi\in\Hom_U(L(\lambda),V\otimes L(\lambda))$ and
$\varphi=\check{R}\circ\psi\in\Hom_U(L(\lambda),L(\lambda)\otimes V)$. Since
$v_\psi$ is of weight $0$, we have
$\varphi(\overline\one_\lambda)=\check{R}\psi(\overline\one_\lambda)=
\check{R}(v_\psi\otimes\overline\one_\lambda)\,+$ lower order terms
$=\overline\one_\lambda\otimes v_\psi\,+$ lower order terms.
\end{proof}

Applying Propositions \ref{prop-u} and \ref{prop-v} to the case $V=F$ we get Theorem
\ref{thm-theta}.

\subsection{Reduced fusion elements}

In this subsection we describe $\Theta^{-1}:
F[0]^{K_\lambda+\widetilde{K}_\lambda}\to\Hom_U(L(\lambda),L(\lambda)\otimes
F)$ explicitly. We are going to obtain a formula for $\Theta^{-1}$ in terms
of the Shapovalov form. Recall that we can regard $\S_\lambda$ as a
bilinear form on $M(\lambda)$. Denote by $\overline{\S}_\lambda$ the
corresponding bilinear form on $L(\lambda)$. For any $\beta\in Q_+$ denote
by $\overline{\S}_\lambda^\beta$ the restriction of
$\overline{\S}_\lambda$ to $L(\lambda)[\lambda-\beta]$. Let
$y_\beta^i\cdot\overline\one_\lambda$ be an arbitrary basis in
$L(\lambda)[\lambda-\beta]$, where $y_\beta^i\in U^-[-\beta]$.

Take $f\in F[0]^{K_\lambda+\widetilde{K}_\lambda}$ and set
$\varphi=\Theta^{-1}(f)$, 
\[
\varphi(\overline\one_\lambda)=
\sum_{\beta\in Q_+}\sum_iy_\beta^i\overline\one_\lambda\otimes f^{\beta,i}.
\]
\begin{remark}
For $\beta=0$ we have $y_\beta^i=1$ and $f^{\beta,i}=f$.
\end{remark}
\begin{proposition}\label{prop-Theta-inv}
$f^{\beta,i}=\sum_j\left(\overline{\mathbb
S}_\lambda^\beta\right)^{-1}_{ij}\overrightarrow{\theta\left(y_\beta^j\right)}f$.
\end{proposition}

\begin{proof}
For any $\beta=\sum_jc_j\alpha_j\in Q_+$ set $k_\beta=\prod_jk_j^{c_j}\in T$
and $\Lambda_\beta=q^\lambda(k_\beta)=
\prod_jq^{d_jc_j\langle\lambda,\alpha_j^\vee\rangle}$.

Set $\xi=\varphi(\overline\one_\lambda)$. Clearly, $\xi$ is a singular
element in $L(\lambda)\otimes F$. In particular, $(k_i\otimes
k_i)\xi=q^{d_i\langle\lambda,\alpha_i^\vee\rangle}\xi$ and $(e_i \otimes 1 +
k_i \otimes e_i)\xi = 0$. Thus $(e_i \otimes
1)\xi=q^{d_i\langle\lambda,\alpha_i^\vee\rangle}(1 \otimes
\sigma^{-1}(e_i))\xi$. By induction we get $(x \otimes 1)\xi=\Lambda_\beta(1
\otimes \sigma^{-1}(x))\xi$ for any $x\in U^+[\beta]$.

Let $\omega'$ be the involutive antiautomorphism of $\U$ given by
$\omega'(e_i)=f_i$, $\omega'(f_i)=e_i$, $\omega'(t_i)=t_i$.
Set $x_\beta^j=\omega'(y_\beta^j)$. Then we have
\begin{equation}\label{eqn-xi}
\left(\overline{\mathbb
S}_\lambda\otimes\id\right)
\left(\overline\one_\lambda\otimes
\left(x_\beta^j\otimes1\right)\xi\right)
=\Lambda_\beta\left(\overline{\S}_\lambda\otimes\id\right)
\left(\overline\one_\lambda\otimes\left(1\otimes
\sigma^{-1}\left(x_\beta^j\right)\right)\xi\right).
\end{equation}

It is easy to show by induction on $\height\beta$ that
$\omega(x_\beta^j)=q^{c(\beta)}y_\beta^jk_\beta$ and
$\sigma^{-1}(x_\beta^j)=q^{c(\beta)}\theta(y_\beta^j)k_\beta^{-1}$ for a
certain $c(\beta)$. (Actually
$c(\beta)=d_1+\ldots+d_l-\frac{1}{2}\langle\beta,d_1\alpha_1^\vee+\ldots+d_l\alpha_l^\vee\rangle$.)
Hence the l.\ h.\ s.\ of (\ref{eqn-xi}) equals
\[
\sum_i\overline{\S}_\lambda\left(\omega\left(x_\beta^j\right)
\overline\one_\lambda,y_\beta^i\overline\one_\lambda\right)f^{\beta,i}=
q^{c(\beta)}\Lambda_\beta\sum_i\overline{\S}_\lambda
\left(y_\beta^j\overline\one_\lambda,y_\beta^i\overline\one_\lambda\right)f^{\beta,i}
\]
and the r.\ h.\ s.\ of (\ref{eqn-xi}) equals
\[
\Lambda_\beta\overrightarrow{\sigma^{-1}\left(x_\beta^j\right)}f=
q^{c(\beta)}\Lambda_\beta\overrightarrow{\theta\left(y_\beta^j\right)}f.
\]
Combining these together we get
\[
\sum_i\overline{\S}_\lambda
\left(y_\beta^j\overline\one_\lambda,y_\beta^i\overline\one_\lambda\right)f^{\beta,i}=
\overrightarrow{\theta\left(y_\beta^j\right)}f,
\]
and the proposition follows.
\end{proof}

\medskip

For any $\lambda\in\h^*$ consider
\begin{equation}\label{J_red}
J^{\red}(\lambda)=\sum_{\beta\in Q_+}\sum_{i,j}\left(\overline{\mathbb
S}_\lambda^\beta\right)^{-1}_{ij}
y_\beta^i\otimes\theta\left(y_\beta^j\right).
\end{equation}
One can regard $J^{\red}(\lambda)$ as an element in a certain completion of
$U^-\otimes U^+$.

\begin{remark}
This element $J^{\red}(\lambda)$ is not uniquely defined (e.g., because
$U^-\to L(\lambda)$ has a kernel), but this does not affect our further
considerations.
\end{remark}

\begin{remark}
For $f\in F[0]^{K_\lambda+\widetilde{K}_\lambda}$ and
$\varphi=\Theta^{-1}(f)$ one has
$\varphi(\overline\one_\lambda)=J^{\red}(\lambda)(\overline\one_\lambda\otimes
f)$.
\end{remark}
Let us define an associative product $\star_\lambda$ on
$F[0]^{K_\lambda+\widetilde{K}_\lambda}$ by means of $\Theta$, i.e., for any
$f_1,f_2\in F[0]^{K_\lambda+\widetilde{K}_\lambda}$ we define
$f_1\star_\lambda f_2=\Theta(\varphi_1\ast\varphi_2)$, where
$\varphi_1=\Theta^{-1}(f_1)$, $\varphi_2=\Theta^{-1}(f_2)$, and $\ast$ is
the product on $\Hom_U(L(\lambda),L(\lambda)\otimes F)$ given by
\eqref{eqn-basic-star-prod}. By this definition, we get a right $\U$-module
algebra $(F[0]^{K_\lambda+\widetilde{K}_\lambda},\star_\lambda)$.

\begin{theorem}\label{14}
We have
\begin{equation}\label{sym_space_star_prod}
f_1\star_\lambda f_2= \mu\left(\overrightarrow{J^{\red}(\lambda)}(f_1\otimes
f_2)\right).
\end{equation}
\end{theorem}

\begin{proof}
Observe that
\begin{gather*}
(\varphi_1\ast\varphi_2)(\overline\one_\lambda)=
(\id\otimes\mu)(\varphi_1\otimes\id)(\varphi_2(\overline\one_\lambda))=\\
(\id\otimes\mu)(\varphi_1\otimes\id)\left(\overline\one_\lambda\otimes
f_2+\sum_{\beta\in
Q_+\setminus\{0\}}\sum_iy_\beta^i\cdot\overline\one_\lambda\otimes
f_2^{\beta,i}\right)=\\
(\id\otimes\mu)\left(\varphi_1(\overline\one_\lambda)\otimes
f_2+\sum_{\beta\in
Q_+\setminus\{0\}}\sum_i(\Delta(y_\beta^i)\varphi_1(\overline\one_\lambda))\otimes
f_2^{\beta,i}\right)=\\
\overline\one_\lambda\otimes\left(f_1f_2+\sum_{\beta\in Q_+\setminus\{0\}}
\sum_i\left(\overrightarrow{y_\beta^i}f_1\right)f_2^{\beta,i}\right)+\mbox{lower
order terms},
\end{gather*}
where in the last equation we use the fact that for any $y\in U^-_+$
we have $\Delta(y)=1\otimes y+\sum_k y_k\otimes z_k$ with $y_k\in U^-_+$.
Therefore
\[
f_1\star_\lambda f_2=f_1f_2+\sum_{\beta\in
Q_+\setminus\{0\}}\sum_i\left(\overrightarrow{y_\beta^i}f_1\right)f_2^{\beta,i}=
\sum_{\beta\in Q_+}\sum_i\left(\overrightarrow{y_\beta^i}f_1\right)f_2^{\beta,i}.
\]
To finish the proof it is enough to apply Proposition
\ref{prop-Theta-inv} to $f_2$.
\end{proof}
\begin{remark}
Theorem~\ref{14} together with results of \cite{KST} implies that the
algebras $\Hom_{\U}(L(\lambda), L(\lambda)\otimes F)$, $(\End
L(\lambda))_{\fin}$, and $\bigl(F[0]^{K_{\lambda}+\widetilde{K}_{\lambda}},
\star_{\lambda}\bigr)$ are isomorphic as right Hopf module algebras over
$\U$.
\end{remark}

\section{Limiting properties of the fusion element}\label{Sect-limit}



We say that $\lambda\in\h^*$ is \textit{generic} if
$\langle\lambda+\rho,\beta^\vee\rangle\not\in\N$ for all $\beta\in\Rb_+$. In
this case $L(\lambda)=M(\lambda)$, and we set
$J(\lambda)=J^{\red}(\lambda)$. Notice that $J(\lambda)$ up to a $U^0$-part
equals the fusion element related to the Verma module $M(\lambda)$ (see,
e.g., \cite{ES_DYBE}).

\subsection{Regularity}\label{Subsect-regularity-general}

Let $\lambda_0\in\h^*$. Since $J(\lambda)$ is invariant w.\ r.\ to
$\tau(\theta\otimes\theta)$ (where $\tau$ is the tensor permutation), one
can easily see that the following conditions on $\lambda_0$ are equivalent:
1) for any $U^-$-module $M$ the family of operators $J(\lambda)^M:M\otimes
F[0]^{\widetilde{K}_{\lambda_0}}\to M\otimes F$ naturally defined by
$J(\lambda)$ is regular at $\lambda=\lambda_0$, 2) for any $U^+$-module $N$
the family of operators $J(\lambda)_N:F[0]^{K_{\lambda_0}}\otimes N\to
F\otimes N$ naturally defined by $J(\lambda)$ is regular at
$\lambda=\lambda_0$. We will say that $\lambda_0$ is {\it $J$-regular} if
these conditions are satisfied. Clearly, any generic $\lambda_0$ is
$J$-regular.

\begin{proposition}
Assume that $\lambda_0\in\h^*$ is $J$-regular.
Then $F[0]^{K_{\lambda_0}}=F[0]^{\widetilde{K}_{\lambda_0}}=
F[0]^{K_{\lambda_0}+\widetilde{K}_{\lambda_0}}$.
\end{proposition}

\begin{proof}
Let $g\in F[0]^{\widetilde{K}_{\lambda_0}}$. If $\lambda\in\h^*$ is generic,
then the element $J(\lambda)^{M(\lambda)}(\mathbf1_\lambda\otimes g)$ is a
singular vector of weight $\lambda$ in $M(\lambda)\otimes F$. Therefore
$Z:=\lim_{\lambda\to\lambda_0}J(\lambda)^{M(\lambda)}(\mathbf1_\lambda\otimes
g)$ is a singular vector of weight $\lambda_0$ in $M(\lambda_0)\otimes F$,
and hence we have $\varphi_Z\in\Hom_U(M(\lambda_0), M(\lambda_0)\otimes F)$,
$\varphi_Z(\one_{\lambda_0})=Z$.

Under the natural map
$\Hom_U(M(\lambda_0), M(\lambda_0)\otimes F)\to\Hom_U(L(\lambda_0), L(\lambda_0)\otimes
F)$ we have $\varphi_Z\mapsto\varphi_{\overline{Z}}$, where
$\varphi_{\overline{Z}}(\overline\one_{\lambda_0})=\overline{Z}=$ the projection
of $Z$ onto $L(\lambda_0)\otimes F$. Now notice that
$g=\Theta(\varphi_{\overline{Z}})\in
F[0]^{K_{\lambda_0}+\widetilde{K}_{\lambda_0}}$.

The proof of $F[0]^{K_{\lambda_0}}=
F[0]^{K_{\lambda_0}+\widetilde{K}_{\lambda_0}}$ is similar.
\end{proof}

\begin{proposition}\label{prop-surj-1-general}
Assume that $\lambda_0\in\h^*$ is $J$-regular.
Then the natural map $\Hom_U(M(\lambda_0), M(\lambda_0)\otimes
F)\to\Hom_U(L(\lambda_0), L(\lambda_0)\otimes F)$ is surjective.
\end{proposition}

\begin{proof}
We have the isomorphism
\[
\Theta:\Hom_U(L(\lambda_0),
L(\lambda_0)\otimes F)\to
F[0]^{K_{\lambda_0}+\widetilde{K}_{\lambda_0}}=F[0]^{\widetilde{K}_{\lambda_0}}.
\]
Now take $g\in F[0]^{\widetilde{K}_{\lambda_0}}$. Consider
$Z=\lim_{\lambda\to\lambda_0}J(\lambda)^{M(\lambda)}(\mathbf1_\lambda\otimes
g)\in M(\lambda_0)\otimes F$. Since $Z$ a singular vector of weight
$\lambda_0$, we have $\varphi_Z\in\Hom_U(M(\lambda_0), M(\lambda_0)\otimes
F)$, $\varphi_Z(\one_{\lambda_0})=Z$. Clearly, under the mapping
$\Hom_U(M(\lambda_0), M(\lambda_0)\otimes F)\to\Hom_U(L(\lambda_0),
L(\lambda_0)\otimes F)$ the image of $\varphi_Z$ equals to $\Theta^{-1}(g)$,
which proves the proposition.
\end{proof}

\begin{proposition}\label{prop-Kostant}
Assume that $\lambda_0\in\h^*$ is $J$-regular. Then the action map
$U_{\fin}\to(\End L(\lambda_0))_{\fin}$ is surjective.
\end{proposition}

\begin{proof}
Recall that 
we have the isomorphisms
\begin{gather*}
\Hom_U (M(\lambda_0),
M(\lambda_0)\otimes F)\simeq(\End M(\lambda_0))_{\fin},\\
\Hom_U (L(\lambda_0), L(\lambda_0)\otimes F)\simeq(\End
L(\lambda_0))_{\fin}.
\end{gather*}
It is well known that the action map $U_{\fin}\to(\End M(\lambda_0))_{\fin}$
is surjective for any $\lambda_0\in\h^*$ (see \cite{Joseph_book,
Joseph_Letzter_Verma}). Since by Proposition
\ref{prop-surj-1-general} the map $(\End M(\lambda_0))_{\fin}\to(\End
L(\lambda_0))_{\fin}$ is surjective, the map $U_{\fin}\to(\End
L(\lambda_0))_{\fin}$ is also surjective.
\end{proof}

\begin{proposition}\label{prop-limit-general}
Assume that $\lambda_0\in\h^*$ is $J$-regular. Then for any $f,g\in
F[0]^{K_{\lambda_0}}$ we have $\overrightarrow{J(\lambda)}(f\otimes
g)\rightarrow \overrightarrow{J^{\red}(\lambda_0)}(f\otimes g)$ as
$\lambda\rightarrow\lambda_0$.
\end{proposition}

\begin{proof}
For any $\lambda\in\h^*$ we may naturally identify $M(\lambda)$ with $U^-$
as $U^-$-modules. Therefore we know by definition of $J$-regularity that
$J(\lambda)^{M(\lambda)}(\mathbf1_\lambda\otimes g)$ is regular at
$\lambda=\lambda_0$. Thus $J(\lambda)^{M(\lambda)}(\mathbf1_\lambda\otimes
g)\rightarrow Z\in M(\lambda_0)\otimes F$ as $\lambda\rightarrow\lambda_0$.
In an arbitrary basis $y_\beta^i\in U^-[-\beta]$ we have
\[
J(\lambda)^{M(\lambda)}(\mathbf1_\lambda\otimes g)=\sum_{\beta\in
Q_+}\sum_{i,j}\left(\S_\lambda^\beta\right)^{-1}_{ij}
y_\beta^i\mathbf1_\lambda\otimes\overrightarrow{\theta\left(y_\beta^j\right)}g,
\]
and
\[
Z=\sum_{\beta\in Q_+}\sum_{i,j}a^\beta_{ij}\cdot
y_\beta^i\mathbf1_{\lambda_0}\otimes
\overrightarrow{\theta\left(y_\beta^j\right)}g
\]
for some coefficients $a^\beta_{ij}$.

Now choose a basis $y_\beta^i\in U^-[-\beta]$ in the following
way: first take a basis in
$K_{\lambda_0}[-\beta]=K_{\lambda_0}\cap U^-[-\beta]$ and then
extend it arbitrarily to a basis in the whole $U^-[-\beta]$. In
this basis the projection $\overline{Z}\in L(\lambda_0)\otimes F$
of the element $Z$ is given by
\begin{equation}\label{eqn-Z-bar}
\overline{Z}=\sum_{\beta\in
Q_+}\sum_{\ y_\beta^i,y_\beta^j\not\in
K_{\lambda_0}[-\beta]}a^\beta_{ij}\cdot
y_\beta^i\overline\one_{\lambda_0}\otimes
\overrightarrow{\theta\left(y_\beta^j\right)}g.
\end{equation}

Now notice that $Z$, being the limit of singular vectors of weight $\lambda$
in $M(\lambda)\otimes F$, defines the intertwining operator
$\varphi_Z\in\Hom_U(M(\lambda_0), M(\lambda_0)\otimes F)$,
$\varphi_Z(\one_{\lambda_0})=Z$. Under the natural map $\Hom_U(M(\lambda_0),
M(\lambda_0)\otimes F)\to\Hom_U(L(\lambda_0), L(\lambda_0)\otimes F)$ we
have $\varphi_Z\mapsto\varphi_{\overline{Z}}$, where
$\varphi_{\overline{Z}}(\overline\one_{\lambda_0})=\overline{Z}$. Therefore
$\overline{Z}=J^{\red}(\lambda_0)^{M(\lambda_0)}(\overline\one_{\lambda_0}\otimes
g)$ by Proposition \ref{prop-Theta-inv} and the definition of
$J^{\red}(\lambda_0)$. Comparing this with \eqref{eqn-Z-bar} we conclude
that for all $i, j$ such that $y_\beta^i,y_\beta^j\not\in
K_{\lambda_0}[-\beta]$ we have $a^\beta_{ij}=\left(\overline{\mathbb
S}_{\lambda_0}^\beta\right)^{-1}_{ij}$.

Finally,
\begin{gather*}
\overrightarrow{J(\lambda)}(f\otimes g)\rightarrow \sum_{\beta\in
Q_+}\sum_{i,j}a^\beta_{ij} \overrightarrow{y_\beta^i}f\otimes
\overrightarrow{\theta\left(y_\beta^j\right)}g=\\
\sum_{\beta\in Q_+}\sum_{\ y_\beta^i,y_\beta^j\not\in
K_{\lambda_0}[-\beta]}a^\beta_{ij}
\overrightarrow{y_\beta^i}f\otimes
\overrightarrow{\theta\left(y_\beta^j\right)}g=\\
\sum_{\beta\in Q_+}\sum_{\ y_\beta^i,y_\beta^j\not\in
K_{\lambda_0}[-\beta]}\left(\overline{\mathbb
S}_{\lambda_0}^\beta\right)^{-1}_{ij} \overrightarrow{y_\beta^i}f\otimes
\overrightarrow{\theta\left(y_\beta^j\right)}g=
\overrightarrow{J^{\red}(\lambda_0)}(f\otimes g)
\end{gather*}
as $\lambda\to\lambda_0$.
\end{proof}

\begin{corollary}\label{cor-limit-general}
Assume that $\lambda_0\in\h^*$ is $J$-regular.
Let $f_1,f_2\in F[0]^{K_{\lambda_0}}$.
Then $f_1\star_\lambda f_2\rightarrow f_1\star_{\lambda_0}f_2$ as
$\lambda\rightarrow\lambda_0$.\qed
\end{corollary}

\subsection{One distinguished root case}\label{Subsect-limit-one-root}


\begin{theorem}\label{thm-reg-one-root-1}
Let $\alpha\in\Rb_+$. Consider $\lambda_0\in\h^*$ that satisfies
$\langle\lambda_0+\rho,\alpha^\vee\rangle=n\in\N$,
$\langle\lambda_0+\rho,\beta^\vee\rangle\not\in\N$ for all
$\beta\in\Rb_+\setminus\{\alpha\}$. Then $\lambda_0$ is $J$-regular.
\end{theorem}

\begin{proof}

Fix an arbitrary line $l\subset\h^*$ through $\lambda_0$,
$l=\{\lambda_0+t\nu\,|\,t\in\mathbb R\}$, transversal to the hyperplane
$\langle\lambda+\rho,\alpha^\vee\rangle=n$.

Identify $M(\lambda)$ with $U^-$ in the standard way. Recall
that we have a basis $y_\beta^i\in U^-[-\beta]$ for $\beta\in
Q_+$. Let $L\left(\S_\lambda^\beta\right)\in\End
U^-[-\beta]$ be given by the matrix $\left(\mathbb
S_\lambda^\beta\right)_{ij}$ in the basis
$y_\beta^i$. Notice that $\Ker L\left(\mathbb
S_{\lambda_0}^\beta\right)=\Ker\mathbb
S_{\lambda_0}^\beta=K_{\lambda_0}[-\beta]=K_{\lambda_0}\cap
U^-[-\beta]$. For any $\lambda\in l$ sufficiently close to
$\lambda_0$, $\lambda\neq\lambda_0$ we have $M(\lambda)$ is
irreducible, and $L\left(\S_\lambda^\beta\right)$ is
invertible for any $\beta\in Q_+$. In this notation we have
\[
J(\lambda)=\sum_{\beta\in Q_+}\sum_j L\left(\mathbb
S_\lambda^\beta\right)^{-1}y_\beta^j\otimes\theta\left(y_\beta^j\right).
\]

Take $\lambda=\lambda_0+t\nu\in l$. Fix any $\beta\in Q_+$ and set
$V=U^-[-\beta]$, $A_t=L\left(\S_\lambda^\beta\right)$, $V_0=\Ker
A_0=K_{\lambda_0}[-\beta]\subset V$. Write $A_t=A_0+tB_t$, where $B_t$ is
regular at $t=0$. Since $J(\lambda)$ may have at most simple poles (see,
e.g., \cite{ESt_2}) we have $A_t^{-1}=\frac{1}{t}C+D_t$, where $D_t$ is
regular at $t=0$.

\begin{lemma}\label{lem-limit-lin-alg-1}
$\Image C\subset V_0$.
\end{lemma}

\begin{proof}
We have $A_tA_t^{-1}=\id$ for any $t\neq0$, i.e.,
$\frac{1}{t}A_0C+A_0D_t+B_tC+tB_tD_t=\id$. Since the left hand
side should be regular at $t=0$, we have $A_0C=0$, which proves
the lemma.
\end{proof}

For $t\neq 0$ set $J_t=\sum_jA_t^{-1}y_j\otimes\theta(y_j)$ (from now on we
are omitting the index $\beta$ for the sake of brevity). By Lemma
\ref{lem-limit-lin-alg-1} we have $Cy_j\in V_0=K_{\lambda_0}[-\beta]$. Hence
for $f\in F[0]^{K_{\lambda_0}}$ we have $\overrightarrow{Cy_j}f=0$.
Therefore
$\overrightarrow{A_t^{-1}y_j}f=\frac{1}{t}\overrightarrow{Cy_j}f+\overrightarrow{D_ty_j}f=
\overrightarrow{D_ty_j}f$. This proves the regularity of $(J_t)_N(f\otimes
\cdot)$ at $t=0$, i.e., the regularity of $J_q(\lambda)_N(f\otimes \cdot)$
at $\lambda=\lambda_0$.
\end{proof}

\subsection{Subset of simple roots case}\label{Subsect-limit-new}


\begin{theorem}\label{thm-reg-subset}
Let $\Gamma\subset\Pi$. Consider $\lambda_0\in\h^*$ that satisfies
$\langle\lambda_0+\rho,\alpha_i^\vee \rangle\in\N$ for all
$\alpha_i\in\Gamma$, $\langle\lambda_0+\rho,\beta^\vee\rangle\not\in\N$
for all $\beta\in\Rb_+\setminus\spanv\Gamma$. Then $\lambda_0$ is
$J$-regular.
\end{theorem}

\begin{proof}
Recall that the only singularities of
$J(\lambda)$ near $\lambda_0$ are (simple) poles on the hyperplanes
$\langle\lambda-\lambda_0,\alpha^\vee\rangle=0$ for $\alpha\in\B
R_+\cap\spanv\Gamma$ (see, e.g., \cite{ESt_2}). Therefore it is enough to
show that for any $f\in F[0]^{K_{\lambda_0}}$ the operator
$J(\lambda)_N(f\otimes \cdot)$ has no singularity at any such hyperplane.

Take $\alpha\in\Rb_+\cap\spanv\Gamma$ and consider a hyperplane
$\langle\lambda-\lambda_0,\alpha^\vee\rangle=0$. Take an arbitrary
$\lambda'\in\h^*$ such that
$\langle\lambda'-\lambda_0,\alpha^\vee\rangle=0$, and
$\langle\lambda'+\rho,\beta^\vee\rangle\not\in\N$ for all $\beta\in\B
R_+\setminus\{\alpha\}$.

\begin{lemma}\label{KsubsetK}
$K_{\lambda'}\subset K_{\lambda_0}$.
\end{lemma}
\begin{proof}
First, assume that $\lambda_0$ is dominant integral, i.e., $\Gamma=\Pi$. Set
$n=\langle\lambda_0+\rho,\alpha^\vee\rangle$. Consider the irreducible
$\U$-module $L(\lambda_0)$ with a highest weight vector
$\overline\one_{\lambda_0}$. The inclusion $K_{\lambda'}\subset
K_{\lambda_0}$ is equivalent to the equality
$K_{\lambda'}\overline\one_{\lambda_0}=0$. By Proposition~\ref{prop-third},
$K_{\lambda'}$ is generated by an element $x_{\lambda'}\in U^-$ of weight
$-n\alpha$. Hence
the vector $x_{\lambda'}\overline\one_{\lambda_0}$ has weight
$\lambda_0-n\alpha$. Thus it suffices to show that $\lambda_0-n\alpha$ is
not a weight of $L(\lambda_0)$. Indeed, easy computations show that
$s_\alpha(\lambda_0-n\alpha)=\lambda_0+\langle\rho,\alpha^\vee\rangle\alpha
> \lambda_0$, and since the set of weights of $L(\lambda_0)$ is
$W$-invariant, the lemma is proved in this case.

Now, let us consider the general case. Let $\lambda_0$ be an arbitrary
weight satisfying the assumptions of the theorem. Set
$n_i=\langle\lambda_0+\rho,\alpha_i^\vee\rangle$. Denote by $K'_{\lambda_0}$
the right ideal in $U^-$ generated by the elements $f_i^{n_i}$,
$\alpha_i\in\Gamma$. By Proposition~\ref{prop-first}, we have
$K'_{\lambda_0}\subset K_{\lambda_0}$.

Denote by $X(\lambda_0)$ the set of dominant integral weights $\mu$ such
that $\langle\mu+\rho,\alpha_i^\vee\rangle=n_i$ for all $\alpha_i\in\Gamma$.
We have proved that $K_{\lambda'}\subset K_\mu$ for any $\mu\in
X(\lambda_0)$. Thus $K_{\lambda'}\subset\bigcap_{\mu\in X(\lambda_0)}
K_\mu$.

For $\mu\in X(\lambda_0)$, Proposition~\ref{prop-fourth} yields that $K_\mu$
is generated by $K'_{\lambda_0}$ and the elements
$f_j^{\langle\mu+\rho,\alpha_j^\vee\rangle}$, $j\in\Pi\setminus\Gamma$.
Notice that choosing $\mu\in X(\lambda_0)$, all the numbers
$\langle\mu+\rho,\alpha_j^\vee\rangle$ can be made arbitrary large. Hence,
$K'_{\lambda_0}=\bigcap_{\mu\in X(\lambda_0)} K_\mu$. Therefore,
$K_{\lambda'}\subset K'_{\lambda_0}\subset K_{\lambda_0}$. The lemma is
proved.
\end{proof}

The lemma implies that $F[0]^{K_{\lambda'}}\supset F[0]^{K_{\lambda_0}}$,
and applying Theorem \ref{thm-reg-one-root-1} to $\lambda'$ we complete the
proof of the theorem.
\end{proof}


\begin{remark}
In fact, it is possible to prove that if $\lambda_0$ satisfies the
assumptions of Theorem~\ref{thm-reg-subset}, then $K_{\lambda_0}$ is
generated by the elements
$f_i^{\langle\lambda_0+\rho,\alpha_i^\vee\rangle}$, $\alpha_i\in\Gamma$.
Indeed, this fact is well-known in the classical case $q=1$, see \cite{Ja}.
Let us consider the highest weight module $V_q=M_q (\lambda_0
)/K'_{\lambda_0}$ (we have used at this point the notation $M_q$ for the
Verma modules over $\Uc$). We have to prove that it is irreducible under
assumptions of the theorem. If not, there exists some weight space $V_q
[\mu]$ such that the determinant of the restriction of the Shapovalov form
on $V_q$ to $V_q [\mu]$ is zero. Let us denote the Shapovalov form on $V_q$
by $\mathbb{S}_q$ and its restriction to $V_q [\mu]$ by $\mathbb{S}_q
[\mu]$. Taking limit $q\to 1$ (it can be done in the same way as in
\cite[Sections~3.4.5--3.4.6]{Joseph_book}) we see that $\mathbb{S}_q
[\mu]\to \mathbb{S}_1 [\mu]$ and $\det \mathbb{S}_q [\mu]\to \det
\mathbb{S}_1 [\mu]$. However, the latter determinant is non-zero because
$V_1$ is irreducible. Hence, $V_q$ is also irreducible.
\end{remark}

\subsection{Application to Poisson homogeneous spaces}\label{Subsect-symm}

Let $\Gamma\subset\Pi$. Assume that $\lambda\in\h^*$ is such that
$\langle\lambda,\alpha^\vee\rangle=0$ for all $\alpha\in\Gamma$, and
$\langle\lambda+\rho,\beta^\vee\rangle\not\in\N$ for all $\beta\in\B
R_+\setminus\spanv\Gamma$. By Theorem \ref{thm-reg-subset}, $\lambda$ is
$J$-regular. In particular, $
F[0]^{K_{\lambda}+\widetilde{K}_{\lambda}}=F[0]^{K_{\lambda}}$.

In what follows it will be more convenient to write $F_q$, $J_q$, and
$K_{q,\lambda}$ instead of $F$, $J$, and $K_\lambda$. We will also need the
classical limits $F_1=\lim_{q\to 1}F_q$ and $K_{1,\lambda}=\lim_{q\to
1}K_{q,\lambda}$. They can be defined in the same way as
in~\cite[Sections~3.4.5--3.4.6]{Joseph_book}.

Clearly, $F_1$ is the algebra of regular functions on the connected simply
connected group $G$, whose Lie algebra is $\g$. Let $\kg$ be a reductive
subalgebra of $\g$ which contains $\h$ and is defined by
$\Gamma$, $K$ the corresponding subgroup of $G$, and $F(G/K)$ the algebra of
regular functions on the homogeneous space $G/K$. According to \cite[Theorem
33]{KST}, we have $F(G/K)=F_1[0]^{K_{1,\lambda}}$. Therefore we get

\begin{proposition}\label{prop-lim}
$\lim_{q\to 1}F_q[0]^{K_{q,\lambda}}=F(G/K)$.\qed
\end{proposition}

Furthermore, since $F_q[0]^{K_{q,\lambda}}$ is a Hopf module algebra over
$\U$, $G/K$ is a Poisson homogeneous space over $G$ equipped with the
Poisson-Lie structure defined by the Drinfeld-Jimbo classical $r$-matrix
$r_0=\sum_{\alpha\in\Rb_+}e_\alpha\wedge e_{-\alpha}$.

All such structures on $G/K$ were described in~\cite{KMST}. It follows
from~~\cite{KMST} that any such Poisson structure on $G/K$ is uniquely
determined by an an intermediate Levi subalgebra $\n$ satisfying
$\kg\subset\n\subset\g$ and some $\lambda\in\h^*$ which satisfies certain
conditions, in particular, $\langle\lambda,\alpha^\vee\rangle=0$ for
$\alpha\in\Gamma$ and $\langle\lambda,\beta^\vee\rangle\not\in\Z$ for
$\beta\in\spanv\Gamma_\n\setminus\spanv\Gamma$. Here $\Gamma_\n$ is the set
of simple roots defining $\n$.

Now we can describe the Poisson bracket on $G/K$ defined by
$\star_\lambda$-multiplication on $F_q[0]^{K_{q,\lambda}}$.

\begin{theorem}\label{thm-lim}
Assume that $\langle\lambda_0,\alpha^\vee\rangle=0$ for $\alpha\in\Gamma$
and $\langle\lambda_0,\beta^\vee\rangle\not\in\Z$ for $\beta\in\B
R_+\setminus\spanv\Gamma$. Then the classical limit of
$(F_q[0]^{K_{q,\lambda_0}},\star_{\lambda_0})$ is the algebra $F(G/K)$ of
regular functions on $G/K$ equipped with the Poisson homogeneous structure
defined by $\n=\g$ and $\lambda_0$.
\end{theorem}

\begin{proof}
By Theorem \ref{thm-reg-subset}, $\lambda_0$ is $J$-regular, i.e., for $f_1,
f_2\in F_q[0]^{K_{q,\lambda_0}}$ we have
\[
f_1\star_{\lambda_0}f_2=\lim_{\lambda\to\lambda_0}f_1\star_{\lambda}f_2=
\mu\left(\overrightarrow{J_q(\lambda)}(f_1\otimes f_2)\right).
\]
Take $q=e^{-\frac{\hbar}{2}}$. Then
$J_q\left(\frac{\lambda}{\hbar}\right)=1\otimes1+\hbar
j(\lambda)+O(\hbar^2)$, and $r(\lambda)=j(\lambda)-j(\lambda)^{21}$ is the
standard trigonometric solution of the classical dynamical Yang-Baxter
equation (see, e.g., \cite{ES_DYBE}). Thus the Poisson bracket on $F(G/K)$
that corresponds to $(F_q[0]^{K_{q,\lambda_0}},\star_{\lambda_0})$ is given
by
\[
\{f_1,f_2\}=\lim_{\lambda\to\lambda_0}\mu\left(\overrightarrow{r(\lambda)}(f_1\otimes
f_2)\right)
\]
for $f_1, f_2\in F(G/K)$. By \cite{KMST}, this is exactly the Poisson
structure defined by $\n=\g$ and $\lambda_0$.
\end{proof}

Notice that an analogous result for simple Lie algebras of classical type
was obtained in \cite{Mudrov} using reflection equation algebras.



\medskip

Proposition \ref{prop-lim} and Theorem \ref{thm-lim} suggest a conjecture
which we formulate below.

Let $G$ be a connected 
Poisson affine algebraic group, $\g$ the corresponding Lie bialgebra with
the co-bracket $\delta$, $X$ a Poisson
homogeneous 
$G$-variety, $Y$ an affine Zariski open dense subset of $X$. Consider the
Poisson algebra $F(Y)$ of regular functions on $Y$. Let $U_q\g$ be a
quantized universal enveloping algebra corresponding to $\g$.

\begin{conjecture}\label{conj-PHS}
There exists a Hopf module algebra over $U_q\g$ whose classical limit is
$F(Y)$.
\end{conjecture}

Let us show another example which confirms this conjecture. Consider the
case $X=G$. Let $D(\g)$ be the classical double of $\g$. According to
\cite{D_poiss_hom_spaces}, Poisson $G$-homogeneous structures on $G$ are in
one-to-one correspondence with Largangian subalgebras of $D(\g)$ transversal
to $\g\subset D(\g)$. Consider such a Lagrangian subalgebra
$\mathfrak{l}\subset D(\g)$, which corresponds to a certain Poisson
$G$-homogeneous structure on $G$. It is well known \cite{D_poiss_lie} that
$\mathfrak{l}$ also induces a new Poisson-Lie structure on $G$, which
differs from the original one by a so-called classical twist. Hence we
obtain a new Lie bialgebra structure $\delta_1$ on the Lie algebra $\g$.

The following conjecture was made in \cite{KPST} and later published in
\cite{KPSST}.

\begin{conjecture}\label{conj-twist}
There exists an element $T$ in a certain completion of $(U_q\g)^{\otimes2}$
which satisfies
\begin{equation}\label{eqn-twist}
T^{12}(\Delta\otimes\id)(T)=T^{23}(\id\otimes\Delta)(T)
\end{equation}
and $(\varepsilon\otimes\id)(T)=(\id\otimes\varepsilon)(T)=1$ such that the
Hopf algebra $U_{q,T}\g$ quantizes $(\g,\delta_1)$. Here $U_{q,T}\g$ and
$U_q\g$ are isomorphic as algebras, and the co-multiplication on $U_{q,T}\g$
is given by $\Delta_T(a)=T\Delta(a)T^{-1}$.
\end{conjecture}

This conjecture was proved in \cite{H, EH}.

Now let $F_q(G)$ be the restricted dual of $U_q\g$. It is well known that
$F_q(G)$ quantizes $F(G)$. Let us equip $F_q(G)$ with a new product defined
by $f_1\star_T f_2= \mu\left(\overrightarrow{T}(f_1\otimes f_2)\right)$.
According to (\ref{eqn-twist}), $\star_T$ is associative. Hence we get

\begin{corollary}
The algebra $(F_q(G), \star_T)$ is a Hopf module algebra over $U_q\g$ which
quantizes the Poisson homogeneous structure on $G$ defined by
$\mathfrak{l}$.\qed
\end{corollary}

\end{sloppy}

\small

\noindent E.K.: Department of Mathematics, Kharkov National University,\\
4 Svobody Sq., Kharkov \,61077, Ukraine\\
e-mail: {\small \tt eugene.a.karolinsky@univer.kharkov.ua}

\medskip

\noindent A.S.: Department of Mathematics, University of G\"oteborg,\\
SE-412 96 G\"oteborg, Sweden\\
e-mail: {\small \tt astolin@math.chalmers.se}

\medskip

\noindent V.T.: St.\,Petersburg Branch of Steklov Mathematical Institute,\\
Fontanka 27, St.\,Petersburg \,191023, Russia;\\
Department of Mathematical Sciences, IUPUI,\\
Indianapolis, IN 46202, USA \\
e-mail: {\small \tt vt@pdmi.ras.ru; vt@math.iupui.edu}

\end{document}